\documentclass{article}
\usepackage[utf8]{inputenc}
\usepackage{amsmath, amsthm, amssymb, amsfonts}
\usepackage{bm}
\usepackage{dsfont}
\usepackage[utf8]{inputenc}
\usepackage{rotate}
\usepackage{tikz}
\usepackage{tikz-cd}
\usepackage[arrow,matrix,curve]{xy} 	
\usepackage{xcolor}
\usepackage{todonotes}
\usepackage{alltt} 
\usepackage{calc}

\theoremstyle{plain}
\newtheorem{theorem}{Theorem}[section]

\newtheorem{corollary}[theorem]{Corollary}

\theoremstyle{definition}

\newcommand{\norm}[1]{\left\lVert#1\right\rVert}

\newcommand{\NN}{{\mathbb{N}}}

\title{On the finite pair correlation function of van der Corput sequences}
\author{Christian Wei\ss{}}
\date{\today}
\begin{document}

\maketitle

\begin{abstract} In this note we derive an explicit formula for the finite empiric pair correlation function $F_N(s)$ of the van der Corput sequence in base $2$ for all $N \in \mathbb{N}$ and $s \geq 0$. The formula can be evaluated without explicit knowledge about the elements of the van der Corput sequence. Moreover, it can be immediately read off that $\lim_{N \to \infty} F_N(s)$ exists only for $0 \leq s \leq 1/2$. 
\end{abstract}

\section{Introduction}

The finite empiric pair correlation function of a sequence $(x_n)_{n \in \mathbb{N}} \in [0,1]$ is for $N \in \mathbb{N}$ and $s \in \mathbb{R}^+_0$ defined by
\begin{align} \label{eq:FN}
F_{N}(s) := \frac{1}{N} \# \left\{ 1 \leq k \neq l \leq N \ : \ \left\| x_k - x_l\right\| \leq \frac{s}{N} \right\},
\end{align}
where $\left\| \cdot \right\|$ is the distance of a number from its nearest integer. It measures the behavior of gaps between the first $N \in \mathbb{N}$ elements of $(x_n)_{n \in \mathbb{N}}$ on a local scale. A sequence $(x_n)_{n \in \mathbb{N}}$ is said to have Poissionian pair correlations if
$F(s):= \lim_{N \to \infty} F_{N}(s) = 2s$ is satisfied for all $s \geq 0$. In order to distinguish it from $F_N(s)$ we call $F(s)$ limiting pair correlation function.\\[12pt]
A generic uniformly distributed random sequence in $[0,1]$ drawn from uniform distribution has Poissonian pair correlations (see e.g. \cite{Mar07} for a proof). Nonetheless, there are only few explicitly known such examples, see \cite{BMV15}, and more recent ones in \cite{LST21} and \cite{LS22}. One of the reasons why such examples are difficult to find is that the gap structure of a sequence is in general hard to describe. Although research mainly focused on the generic (Poissonian) case, also the non-generic case attracts more attention in recent time: in \cite{MS13}, it is shown that the gap distribution of $(\left\{\log(n)\right\})_{n \in \mathbb{N}}$, where $\{ \cdot \}$ denotes the fractional part of a number, has an explicit distribution which is not Poissonian. In \cite{Wei23}, the limiting pair correlation function of $\left( \left\{ \frac{\log(2n-1)}{\log(2)} \right\} \right)_{n \in \mathbb{N}}$ is explicitly calculated by exploiting the simple gap structure of this sequence. Another result in \cite{Lut20} describes the limiting pair correlation function of orbits of a point in hyperbolic space under the action of a discrete subgroup.  Finally, in \cite{Say23}, the non-generic pair correlation statistic of the sequence $({n^\alpha})_{n \in \mathbb{N}}$ is studied for $0 < \alpha < 1$.\\[12pt]
In this note, we add to the growing body of literature by calculating for all $N \in \mathbb{N}$ and $s \geq 0$ the finite empiric pair correlation function $F_N(s)$ of a van der Corput sequence in base $2$, which is a classical example of a low-discrepancy sequence and widely discussed in the literature, see e.g. \cite{KN74}.\\[12pt]
Recall that for an integer $b \geq 2$ the $b$-ary representation of $n \in \NN$ is $n = \sum_{j=0}^\infty e_j b^j$ with $0 \leq e_j = e_j(n) < b$. The radical-inverse function is defined by $g_b(n)=\sum_{j=0}^\infty e_j b^{-j-1}$ for all $n \in \NN$ and the van der Corput sequence in base $b$ is given by $x_i := g_b(i-1)$ for $i \geq 2$. For convenience, we add $x_1 = 0$ as the first element of a van der Corput sequence because it simplifies the presentation of results in our context. 
\begin{theorem} \label{thm:main}
    Let $N \in \mathbb{N}, s \geq 0$ be arbitrary and let its $2$-ary representation of $N$ have the coefficients $e_0, e_1, \ldots, e_M \in \left\{ 0,1\right\}$. Then for the van der Corput sequence $(x_n)_{n \in \mathbb{N}}$ in base $b=2$ we have
    \begin{align} \label{eq} F_N(s) = \frac{1}{N} \sum_{k=0}^M e_k \left( \left\lfloor \frac{s}{N} 2^k \right\rfloor + \sum_{l=k+1}^{N} e_l \cdot 2 \cdot \left\lceil \frac{\left\lfloor\frac{s}{N} 2^{l+1}\right\rfloor}{2} \right\rceil  \right) 2^{k+1}. \end{align}  
\end{theorem}
To the best of our knowledge, this constitutes the first example of an exact closed-form expression of $F_N(s)$ for all $N \in \mathbb{N}$ and all $s \geq 0$, where the right hand side does not rely on explicit knowledge of the involved sequence. Moreover, the expression on the right hand side is surprisingly simple. Finally, the formula is superior in terms of running time because the time needed to evaluate the empiric pair correlation function grows quadratically in $N$, while the running time of the right hand side only grow logarithmically. For example, it would be almost infeasible to calculate the empiric pair correlation function for a given $s > 0$ and $N = 10^9$ on a standard computer via \eqref{eq:FN}, while the evaluation of \eqref{eq} takes less than a second.\\[12pt]
The main ingredient for proving Theorem~\ref{thm:main} is to decompose the set 
$$\left\{ 1 \leq k \neq l \leq N \ : \ \left\| x_k - x_l\right\| \leq \frac{s}{N} \right\}$$
into several subsets, where the indices of the elements of $(x_n)_{n \in \mathbb{N}}$ depend on powers of $2$ instead of $N$. This idea goes back to \cite{WS22}, where the weak limiting pair correlation function of van der Corput sequences was calculated. In principle, our proof technique could be applied to van der Corput sequences in arbitrary base but the expression on the right hand-side of \eqref{eq} would be much longer and more complex. Therefore, we decided here to restrict to the case where a short formula can be given. The reason why this can only be done in base $b=2$ is that there are only two different gap lengths for all $N \in \mathbb{N}$ then while there are up to three different gap lengths for all bases $b \geq 3$. \\[12pt]
In order to give an application of the formula, we show that the limit $\lim_{N \to \infty} F_N(s)$ only exists for $0 \leq s \leq \tfrac{1}{2}$.
\begin{corollary} \label{cor} The limit $\lim_{N \to \infty} F_N(s)$ exists if and only if $0 \leq s \leq \tfrac{1}{2}$. In this case $F(s) = 0$ holds. 
\end{corollary}

\paragraph{Acknowledgements.} Research on this article was conducted during a stay at the Universit\'{e} de Montr\'{e}al, whom the author would like to thank for their hospitality. 

\section{Proof of results}
At first we prove our main result for the van der Corput sequence in base $b = 2$ by applying the same decomposition of \eqref{eq:FN} into subsets as in \cite{WS22}. 
\begin{proof}[Proof of Theorem~\ref{thm:main}]
     Let us write the empiric pair correlation function as
\begin{align*} N \cdot F_N(s) & = \# \underbrace{\left\{ \norm{x_i-x_j} \leq \frac{s}{N} \, : \, 1 \leq i \neq j \leq 2^M  \right\}}_{=:A(s,M,N)} \\
& + 2 \# \underbrace{\left\{ \norm{x_i-x_j} \leq \frac{s}{N} \, : \, 1 \leq i  \leq 2^M, 2^{M} + 1 \leq j \leq N \right\}}_{=:B(s,M,N)} \\
& + \# \underbrace{\left\{ \norm{x_i-x_j} \leq \frac{s}{N} \, : \, 2^M+1 \leq i \neq j \leq N  \right\}}_{=:C(s,M,N)} \ 
\end{align*}
Since the set $A(s,M,N)$ consists of all point $x_i$ which are numbers of the form $\tfrac{k}{2^M}$ with $0 \leq k < 2^{M}$ its magnitude can immediately be calculated as
$$\# A(s,M,N) = \left\lfloor \frac{s}{N} 2^M \right\rfloor \cdot 2^M \cdot 2 =: a(s,M,N).$$
In the definition of the set $B(s,M,N)$, the $x_i$ are again of the form $\tfrac{k}{2^M}$ with $0 \leq k \leq 2^{M}$ while the $x_j$ all have the form $\frac{l}{2^{M+1}}$ with odd $1 \leq l < 2^{M+1}$ by the definition of van der Corput sequences. Hence it follows that
$$\# B(s,M,N) = \left\lceil \frac{\left\lfloor\frac{s}{N} 2^{M+1}\right\rfloor}{2} \right\rceil \cdot (N-2^M) \cdot 2 =: b(s,M,N).$$
Thus, it only remains to calculate $C(s,M,N +)$. To do that, we at first see that $(x_j)_{j=2^M+1}^N$ is the van der Corput sequence $(x_j)_{j=<1}^{N-2^{M}+1}$ translated to the right by $2^{-(M+1)}$. Note that $\norm{x_i-x_j}$ is invariant under simultaneous translation of $x_i$ and $x_j$. If we treat the simpler situation $N = 2^M + 2^k$ with $k < M$, then we can proceed inductively and apply the formula for $A(s,M,N)$ which yields
$$\# C(s,M,N) = \left\lfloor \frac{s}{N} 2^k \right\rfloor \cdot 2^k \cdot 2,$$
because the set of type $B$ is empty here. In the general case we obtain
$$C(s,M,N) = \sum e_k \cdot \left( a(s,k,N) + 2 \cdot b\left(s,k,N-\sum_{l=k+1}^M e_l 2^l\right) \right), $$
where the factor $2$ appears because sets of type $B$ are counted twice. Plugging in the corresponding expressions for $A(\cdot)$ and $B(\cdot)$ yields a sum of the form
    $$\sum_{k=1}^N e_k \left(  \left\lfloor \frac{s}{M} 2^k \right\rfloor 2^{k+1} + 4 \left\lceil \frac{\left\lfloor\frac{s}{M} 2^{k+1}\right\rfloor}{2} \right\rceil \sum_{l=1}^{k-1} e_l 2^l  \right).$$
Collecting powers of $2$, yields the formula on the right hand side of \eqref{eq}.
\end{proof}
Having now formula \eqref{eq} at hand, it is not hard to calculate the limiting behavior of $F_N(s)$.
\begin{proof}[Proof of Corollary~\ref{cor}]
 For $2^{M} \leq N < 2^{M+1}$, the $2$-ary representation of $N \in \mathbb{N}$ is of the form $N = 2^M + \sum_{k=1}^M e_k 2^k$. Thus, 
$$\left\lfloor \frac{s}{N} 2^{N+1} \right\rfloor = 0$$
for all $0 \leq s \leq \tfrac{1}{2}$ and the limit is $F_N(s) = 0$ by Theorem~\ref{thm:main}. Now let $s \in [l, l  + 1)$ for some $l \in \mathbb{N}_0$. Then, we have $F_N(s) = 2l$ for $N = 2^M$, again by Theorem~\ref{thm:main}. If $s \in (1/2 + l, 1/2 + l + 1]$ for some $l \in \mathbb{N}_0$, we choose $N$ big enough such that $\tfrac{2^{M+1}}{2^M+1} \cdot s > 2$. Then $F_N(s) \geq \tfrac{2^{M+1}}{2^M+1}$ for all $N = 2^{M+k} + 2^k$ with $k \in \mathbb{N}$. Thus, $F_N(s)$ cannot converge for $s > \tfrac{1}{2}$.
\end{proof}

\bibliographystyle{alpha}
\bibdata{literatur}
\bibliography{literatur}

\textsc{Ruhr West University of Applied Sciences, Duisburger Str. 100, D-45479 M\"ulheim an der Ruhr,} \texttt{christian.weiss@hs-ruhrwest.de}

\end{document}